\documentclass [12pt,oneside]{amsart}
\usepackage{color}
\usepackage{amssymb,amsmath,amsthm,amsfonts}
\usepackage{setspace,bbm}
\usepackage{empheq}
\usepackage[all]{xy}
\usepackage[bookmarks,colorlinks,breaklinks]{hyperref}  
\hypersetup{linkcolor=blue,citecolor=orange,filecolor=blue,urlcolor=blue} 

\newcommand{\bC}{\mathbb{C}}

\newcommand{\bZ}{\mathbb{Z}}
\newcommand{\bQ}{\mathbb{Q}}
\newcommand{\bR}{\mathbb{R}}
\newcommand{\bN}{\mathbb{N}}

\newcommand{\gcdd}{{\rm gcd}}

\newcommand{\la}{\lambda}

\newtheorem{theorem}{Theorem}[section]
\newtheorem{proposition}[theorem]{Proposition}

\newtheorem{lemma}[theorem]{Lemma}
\newtheorem{claim}[theorem]{Claim}

\newtheorem{definition}[theorem]{Definition}
\newtheorem{remark}[theorem]{Remark}
\newtheorem{corollary}[theorem]{Corollary}

\newtheorem{example}[theorem]{Example}
\definecolor{myblue}{rgb}{0.6, 0.9, 1}

\definecolor{mygreen}{rgb}{0,0,1}
\definecolor{orange}{rgb}{0.8,0,0.2}
\newcommand{\al}{\alpha}

\begin{document}
\author{Niki Myrto Mavraki}
\title{Impossible intersections in a Weierstrass family of elliptic curves}
\address{Niki Myrto Mavraki, Department of Mathematics, University of British Columbia, Vancouver, BC V6T 1Z2, Canada}
\email{\href{mailto:myrtomav@math.ubc.ca}{myrtomav@math.ubc.ca}}
\begin{abstract}
Consider the Weierstrass family of elliptic curves $E_{\la}:y^2=x^3+\la$ parametrized by nonzero $\la\in\overline{\bQ_2}$, and let $P_{\la}(x)=(x,\sqrt{x^3+\la})\in E_{\la}$. In this article, given $\al,\beta\in\overline{\bQ_2}$ such that $\frac{\al}{\beta}\in\bQ$, we provide an explicit description for the set of parameters $\la$ such that $P_{\la}(\al)$ and $P_{\la}(\beta)$ are simultaneously torsion for $E_{\la}$. In particular we prove that the aforementioned set is empty unless $\frac{\al}{\beta}\in\{-2,-\frac{1}{2}\}$. Furthermore, we show that this set is empty even when $\frac{\al}{\beta}\notin\bQ$ provided that $\al$ and $\beta$ have distinct $2-$adic absolute values and the ramification index $e(\bQ_2(\frac{\al}{\beta})~\vert~\bQ_2)$ is coprime with $6$. We also improve upon a recent result of Stoll concerning the Legendre family of elliptic curves $E_{\la}:y^2=x(x-1)(x-\la)$, which itself strengthened earlier work of Masser and Zannier by establishing that provided $a,b$ have distinct reduction modulo $2$, the set $\{\la\in\mathbb{C}\setminus\{0,1\}~:~(a,\sqrt{a(a-1)(a-\la)}),(b,\sqrt{b(b-1)(b-\la)})\in (E_{\la})_{tors}\}$ 
is empty.
\end{abstract}


\maketitle

\section{Introduction}

Let $E_{\la}:y^2=x(x-1)(x-\la)$ be the Legendre family of elliptic curves parametrized by $\la\in\mathbb{C}\setminus\{0,1\}$ and let $P_{\la}=(2,\sqrt{2(2-\la)}),Q_{\la}=(3,\sqrt{6(3-\la)})\in E_{\la}$. Masser and Zannier \cite{M-Z-0,M-Z-1} proved that the set of parameters $\la$ such that both $P_{\la}$ and $Q_{\la}$ are torsion for $E_{\la}$ is finite. Later, in \cite{M-Z-2} they strengthened the result and showed that $2$ and $3$ are not special. More specifically, they proved that for any $P_{\la},Q_{\la}\in E_{\la}$ with $x$ coordinates in $\overline{\mathbb{C}(\la)}$ there are only finitely many $\la$ such that both $P_{\la}$ and $Q_{\la}$ are torsion for $E_{\la}$ unless there exist $n,m\in\bZ$, not both $0$, such that $[m]P_{\la}=[n]Q_{\la}$ for all $\la\in\mathbb{C}$. Moreover, they proved similar finiteness results for any fibred product of two elliptic curves \cite{M-Z-3}. 

Recently, Stoll \cite{Stoll14} proved that in the case of the Legendre family of elliptic curves, given two sections with $x-$coordinates $\al\in\overline{\bQ}$ and $\beta\in\overline{\bQ}$ that have different reductions `modulo 2', the only possible parameters $\la$ such that both are torsion for $E_{\la}$ are $\al$ and $\beta$. Thus, Stoll improved upon results in \cite{M-Z-0,M-Z-1}. The approach from \cite{Stoll14} involves a careful analysis of the $2-$adic behavior of the $n-$th reduced division polynomial of $E_{\la}$. Furthermore, this approach provides a partial result towards the characterization of the set of parameters $(\mu,\la)$ for which three different points are torsion for $E_{\mu,\la}:y^2=x^3+\mu x+\la$, assuming $\la$ is integral at $2$. In \cite[Proposition 7]{Stoll14} he describes the $2-$adic behavior of the $n-$th reduced division polynomial of $E_{\mu,\la}$. However when $\mu=0$ this result does not give precise information on the parameters $\la$ for which a point with constant $x-$coordinate is torsion for $E_{0,\la}:y^2=x^3+\la$. It is primarily this situation that we aim to address here.

In this article we are mainly interested in a Weierstrass family of elliptic curves $E_{\la}:y^2=x^3+\la$, parametrized by $\la\in\bC_2\setminus\{0\}$, where $\mathbb{C}_2$ denotes the completion of $\overline{\bQ_2}$ with respect to the $2-$adic absolute value. Letting $T(\al)$ denote the set of all parameters $\la\in \mathbb{C}_2$ such that $(\al,\sqrt{\al^3+\la})$ is torsion for $E_{\la}$, we establish the following theorem which is one of our main results.

\begin{theorem}\label{mainthm}
If $\al,\beta\in\overline{\bQ_2}\setminus\{0\}$ are such that $\frac{\al}{\beta}\in\bQ\setminus\{-2,-\frac{1}{2}\}$, then $T(\al)\cap T(\beta)=\emptyset$. Moreover, for all $a\in\overline{\bQ_2}\setminus\{0\}$ we have $T(a) \cap T(-2a)=\{-a^3\}$.
\end{theorem}

In order to derive Thoerem \ref{mainthm}, we study the 2-adic absolute values of the elements in $T(\al)$. Our methods are dynamical; we work with an associated family of Latt\`es maps on $\mathbb{P}^1$, taking a quotient of the multiplication by-2 map on $E_{\la}$. With this approach, in Theorem \ref{2empty} and Corollary \ref{1/2empty}, we present an alternative proof and minor strengthening of Stoll's result concerning the Legendre family of elliptic curves. Furthermore, the method applies to other families of rational maps on $\mathbb{P}^1$, which we illustrate with a non-Latt\`es example, $f_{\la}(z)=\frac{z^d+\la}{pz}$, for integer $d\ge 2$ and prime $p\in\bZ$.

\begin{theorem}\label{notLattes}
Let $p\in\bZ$ be a prime and consider the natural reduction map $\rho: \mathbb{P}^1(\bC_p)\to \mathbb{P}^1(\overline{\mathbb{F}_p})$. For $d\in\bZ_{\ge 2}$, let $f_{\lambda}(z)=\frac{z^d+\la}{pz}$. If $\al,\beta\in\mathbb{C}_p\setminus\{0\}$ are such that $\rho(\al)\neq\rho(\beta)$, then there is no parameter $\lambda\in\bC_p$ for which $\alpha$ and $\beta$ are both preperiodic for $f_{\lambda}$.
\end{theorem}

To put our results in the appropriate context, we will highlight some key features of earlier work. Masser and Zannier's original approach in \cite{M-Z-0,M-Z-1,M-Z-2} involved a key recent result by Pila and Zannier \cite{P-Z} and relied strongly on the existence of the analytic uniformization map for an elliptic curve. They further pointed out a dynamical reformulation of the question based on the fact that a point with $x$ coordinate $a$ is torsion for $E_{\la}$ if and only if $a$ is a preperiodic point for the Latt\`{e}s map induced by the multiplication by $2$ map in $E_{\la}$ (see \ref{Lattes} for the definition of a Latt\`es map). Using this reformulation and the equidistribution results of \cite{Baker-Rumely06,CL,favre-rivera06}, DeMarco, Wang and Ye \cite{DWY13} generalized the aforementioned result for points of small canonical height. Also, motivated by these results and replacing the family of Latt\`{e}s maps by other families of rational maps, many results concerning the finiteness of the set of parameters such that both $a$ and $b$ are preperiodic for a $1-$parameter family of rational maps have appeared in \cite{Baker-DeMarco,GHT-ANT,GHT:preprint}. For an overview on the motivation for these results and an outline of the key ideas in the proofs, we refer the reader to \cite{Zannier-book}.

As opposed to the approach in \cite{M-Z-0,M-Z-1,M-Z-2} which uses a key recent result by Pila and Zannier \cite{P-Z} and relies strongly on the existence of the analytic uniformization map for elliptic curves, and the approach from \cite{Baker-DeMarco,DWY13,GHT-ANT,GHT:preprint}, which uses the powerful equidistribution statements of Baker-Rumely \cite{Baker-Rumely06}, Yuan \cite{Yuan} and Yuan-Zhang \cite{Yuan-Zhang} for points of small height, our method, as outlined next, is much simpler.

The structure of this article is as follows. In Section \ref{1}, we consider a Weierstrass family given by $E_{\lambda}:y^2=x^3+\lambda$ where $\lambda\in \mathbb{C}_2\setminus \{0\}$. We use the Latt\`es map $f_{\la}(z)=\frac{z^4-8\lambda z}{4(z^3+\lambda)}$ induced by the duplication map on $E_{\la}$. More specifically,
using the fact that
\begin{align*}
P_{\la}(\al):=(\al,\sqrt{\al^3+\la}) \text{ is torsion for } E_{\lambda} \Longleftrightarrow\alpha \text{ is a preperiodic point for }f_{\lambda}(z)=\frac{z^4-8\lambda z}{4(z^3+\lambda)},
\end{align*}
in Theorem \ref{trichotomy}, we provide a relation between the $2-$adic absolute values of $\la$ and $\al$ for $\la\in\mathbb{C}_2$ such that $P_{\la}(\al)$ is torsion for $E_{\la}$. This relation strongly depends on whether $0,\infty$ which are both $2-$adically attracting fixed points for all $f_{\lambda}$, belong to the orbit of $\al$ under $f_{\la}$. Furthermore, it is a useful step towards finding pairs $(\al,\beta)\in\mathbb{C}_2^2$ such that $T(\al)\cap T(\beta)=\emptyset$, which is what we consider subsequently. Using Theorem \ref{trichotomy}, we establish results of this flavour in Corollary \ref{difempty}. As a special case, we get that if $\al,\beta\in\bQ^{unr}_2$ have distinct $2-$adic absolute values and $\frac{\al^3}{\beta^3}\notin\{-8,-\frac{1}{8}\}$, then $T(\al)\cap T(\beta)=\emptyset$. An important property of the family of elliptic curves and corresponding Latt\`es maps that we exploit in the proof of Theorem \ref{trichotomy}, is that they are isotrivial (see \ref{isotrivial} and \ref{isotrivialEC}.). 

In Section \ref{2}, building on the results obtained in Section \ref{1}, we proceed to establish Theorem \ref{mainthm}. More precisely, we use Lemma \ref{toempty} to reduce the question to proving the coprimality of certain families of polynomials, which can be done by elementary means.

In Section \ref{3}, to further demonstrate the efficacy of our approach, as well as the fact that our method does not rely on isotriviality in general, we apply our method to give a shorter proof of Stoll's result. As a by-product of our method, in Theorem \ref{2empty} and Corollary \ref{1/2empty}, we obtain a slight strengthening of Stoll's original result, \cite[Corollary 4]{Stoll14}. \footnote{While this article was under review, Stoll, in \cite[Proposition 7]{Stoll14v2}, obtained a stronger result that covers our Theorem \ref{2empty} and Corollary \ref{1/2empty} by incorporating elements of our approach in his proof.}

Finally, we conclude with a discussion on other families of maps which are not Latt\`es in Section \ref{4}, where we establish Theorem \ref{notLattes}.

\section{A Weierstrass family: a trichotomy and some impossible intersections}\label{1}

Let $\mathbb{C}_2$ denote the completion of $\overline{\bQ_2}$ with respect to the $2-$adic absolute value, and let $E_{\la}:y^2=x^3+\la$ where $\la\in\mathbb{C}_2\setminus\{0\}$. 

Before proceeding to our results, we will give some definitions. First, we will define the notion of a Latt\`es map, which plays an important role for our purposes. For a survey on its various remarkable properties, we refer the reader to \cite{Milnor} and \cite[6]{Silverman07}. 

\begin{definition}\cite[Section 6.4]{Silverman07}\label{Lattes}
A rational map $\phi:\mathbb{P}^1\to\mathbb{P}^1$ of degree $d\ge 2$ is called a \emph{Latt\`es map} if there are an elliptic curve $E$, a morphism $\psi: E\to E$, and a finite separable morphism $\pi: E\to \mathbb{P}^1$ such that the following diagram is commutative.
$$\begin{array}[c]{ccc}
E&\stackrel{\psi}{\longrightarrow}&E\\
\downarrow\scriptstyle{\pi}&&\downarrow\scriptstyle{\pi}\\
\mathbb{P}^1&\stackrel{\phi}{\longrightarrow}&\mathbb{P}^1.
\end{array}$$
\end{definition}
 
In the following, we will use the Latt\`es maps induced by the multiplication by $2$ on $E_{\lambda}$ for $\lambda\neq 0$. More precisely, we will use the maps $\mathbf{f_{\lambda}}$ defined by the commutative diagram below.
$$\begin{array}[c]{ccc}
E_{\lambda}&\stackrel{[2]}{\longrightarrow}&E_{\lambda}\\
\downarrow\scriptstyle{\pi}&&\downarrow\scriptstyle{\pi}\\
\mathbb{P}^1&\stackrel{\mathbf{f_{\la}}}{\longrightarrow}&\mathbb{P}^1.
\end{array}$$
Here $\pi: E_{\lambda}\to \mathbb{P}^1$ is the projection map onto the $x-$coordinate of the elliptic curve $E_{\la}$ and $\mathbf{f_{\lambda}}:\mathbb{P}^1\to\mathbb{P}^1$ is given as $\mathbf{f_{\la}}([X:Y]) = [X^4 - 8\la XY^3 : 4Y(X^3+ \la Y^3)]$. We will mainly work with the de-homogenized version of $\mathbf{f_\la}$ defined as $f_{\la}(z)=\frac{z^4-8\lambda z}{4(z^3+\lambda)}$, and in this setting we identify $\infty$ with the point $[1:0]\in\mathbb{P}^1$.

 Note that the family of elliptic curves $\{E_{\lambda}\}_{\{\lambda\in\mathbb{C}_2\setminus\{0\}\}}$ may also be viewed as a single elliptic curve $E:y^2=x^3+t$ defined over the function field $\bC_2(t)$, and the family of Latt\`es maps $\{f_{\lambda}\}_ {\{\lambda\in\mathbb{C}_2\setminus\{0\}\}}$ may  also be viewed as a single rational map $f\in\bC_2(t)(z)$. One important property of this elliptic curve $E$ and this rational function $f$ that will aid to the proof of Theorem \ref{mainthmr}, is that they are isotrivial.

\begin{definition}\label{isotrivialEC}
Let $K$ be an algebraically closed field. An elliptic curve $E$ defined over the function field $K(t)$ is called \emph{isotrivial} if there exists a finite extension $L$ of $K(t)$ and an elliptic curve $\mathcal{E}$ defined over $K$ such that $E$ is $L-$isomorphic to $\mathcal{E}$.
\end{definition}

\begin{example}
For the elliptic curve $E:y^2=x^3+t$ over $\mathbb{C}_2(t)$ and $\mathcal{E}:y^2=x^3+1$ over $\mathbb{C}_2$, we have a $\mathbb{C}_2(t^{1/2},t^{1/3})-$isomorphism as follows.
\begin{align*}
\mathcal{E}&\to E\\
(x,y)&\mapsto (xt^{1/3},yt^{1/2}).
\end{align*}
\end{example}

\begin{definition}\label{isotrivial}
Let $K$ be an algebraically closed field. A rational function $\phi\in K(t)(z)$ is called \emph{isotrivial} if there exists a finite extension $L$ of $K(t)$ and a M\"obius map $M\in \mathrm{PGL}(2,L)$ such that $M^{-1}\circ \phi\circ M\in K(z)$.
\end{definition}

\begin{example}\label{iso2}
For the rational function $f(z)=\frac{z^4-8tz}{4(z^3+t)}\in\bC_2(t)(z)$, we have that if $M(z)=t^{1/3}z\in \mathrm{PGL}(2,\mathbb{C}_2(t^{1/3}))$ then $$M^{-1}\circ f\circ M(z)=f_1(z)=\frac{z^4-8z}{4(z^3+1)}\in\mathbb{C}_2(z).$$ 
\end{example}

This equation reflects the fact that all $E_{\la}$, $\la\neq 0$ are isomorphic, see \cite[Theorem 6.46]{Silverman07}

We denote the $2-$adic absolute value defined on $\mathbb{C}_2$ by $|\cdot|$. With this notation $$\mathbf{|2|=\frac{1}{2}}.$$

Moreover, we note that throughout this article we assume that $$\mathbf{0}\boldsymbol\in\mathbb{N}.$$

Let $\al,\la\in\bC_2$. We denote the orbit of $\al$ under the action of $f_{\la}$ by 
\begin{align*}
\mathcal{O}_{f_{\la}}(\al)=\{f^n_{\la}(\al)~:~n\in\bN\}.
\end{align*}
Here, and in general in this article we write $g^n$ for the $n-$th compositional iterate of a function $g$. Moreover, for $g\in\bC_2(z)$, we write 
\begin{align*}
\mathrm{PrePer}(g)&=\{z\in\bC_2~:~ z\text{ is preperiodic for } g\}\\
&=\{z\in\bC_2~:~ \mathcal{O}_{g}(z)\text{ is a finite set}\}.
\end{align*}
Recall that 
\begin{align*}
T(\al)&=\{\la\in\mathbb{C}_2\setminus\{0\}~:~(\al,\sqrt{\al^3+\la})\in(E_{\la})_{tors}\}\\
&=\{\la\in\mathbb{C}_2\setminus\{0\}~:~\al\text{ is preperiodic for }f_{\la}\}.
\end{align*}
Finally, we note that $0,\infty$ are persistently preperiodic points for the family of rational maps $f_{\la}$, $\la\in\mathbb{C}_2\setminus\{0\}$. Thus, in what follows we assume that $$\boldsymbol\alpha\mathbf{\neq0},\boldsymbol\infty.$$

\subsection{A trichotomy}
\begin{theorem}\label{trichotomy}
Let $\la\in T(\al)$. Then either $\la\in\left\{-\al^3,\frac{\al^3}{8}\right\}$ or
\begin{align*}
|\la|\in\{4|\al|^3,4^{1-(1/4)^m}|\al|^3,2^{2+(1/4)^m}|\al|^3~:~m\in\bN_{\ge 1}\}.
\end{align*}
Moreover, exactly one of the following is true.
\begin{enumerate}
\item
$|\la|=4|\alpha|^3$ $\Longleftrightarrow$ $0,\infty\notin\mathcal{O}_{f_{\la}}(\al)$.
\item
$|\la|=4^{1-(1/4)^m}|\al|^3$ for some $m\in\mathbb{N}_{\ge1}$, or $\la=-\al^3$  $\Longleftrightarrow$ $\infty\in\mathcal{O}_{f_{\la}}(\al)$.
\item
$|\la|=2^{2+(1/4)^m}|\al|^3$ for some $m\in\mathbb{N}_{\ge1}$, or $\la=\frac{\al^3}{8}$  $\Longleftrightarrow$ $0\in\mathcal{O}_{f_{\la}}(\al)$.
\end{enumerate}
\end{theorem}

The isotriviality of $f(z)=\frac{z^4-8tz}{4(z^3+t)}\in\bC_2(t)(z)$ will play an important role in the proof of this theorem. We find it worthwhile to point out that if
$L(z)=(4t)^{1/3}z\in \mathrm{PGL}(2,\mathbb{C}_2(t^{1/3}))$ then 
\begin{align}\label{goodreduction}
L^{-1}\circ f\circ L(z)=\frac{z^4-2z}{4z^3+1}=g(z)\in\mathbb{C}_2(z).
\end{align}
The map $g$ here is the Latt\`es map corresponding to the multiplication by $2$ on the elliptic curve $y^2=x^3+\frac{1}{4}$, and has the property that it exhibits $2-$adic good reduction (see \cite[Section 2.5]{Silverman07} for definition). 

For the rest of this section, we write
\begin{align*}
g(z)=\frac{z^4-2z}{4z^3+1}\in\bC_2(z)\text{ as in \eqref{goodreduction}}. 
\end{align*}

Thoerem \ref{trichotomy} will be a consequence of the following proposition.

\begin{proposition}\label{trichotomyg}
Let $w\in\mathrm{Preper}(g)\setminus\{0,\infty\}$. Then, exactly one of the following holds.
\begin{enumerate}
\item
$|w|=1$ $\Longleftrightarrow$ $0,\infty\notin\mathcal{O}_{g}(w)$.
\item
$|w|=|4^{-1/3}|^{\frac{1}{4^m}}$ for some $m\in\bN_{\ge 1}$, or $w^3=-\frac{1}{4}$ $\Longleftrightarrow$ $\infty\in\mathcal{O}_{g}(w)$.
\item
$|w|=|2^{1/3}|^{\frac{1}{4^m}}$ for some $m\in\bN_{\ge 1}$, or $w^3=2$ $\Longleftrightarrow$ $0\in\mathcal{O}_{g}(w)$.
\end{enumerate}
\end{proposition}

Assume for the moment that the aforementioned proposition holds, for the sake of establishing Theorem \ref{trichotomy}.

\begin{proof}[Proof of Theorem \ref{trichotomy}]

The proof is a consequence of the isotriviality of $f\in\bC_2(t)(z)$ and Proposition \ref{trichotomyg} as follows. In view of \eqref{goodreduction}, we get that for all $n\in\bN$ and $\la,\al\in\bC_2$
\begin{align*}
f^n_{\la}(\al)=(4\la)^{1/3}g^n\left(\frac{\al}{(4\la)^{1/3}}\right).
\end{align*}
This implies $\mathcal{O}_{f_{\la}}(\al)=(4\la)^{1/3}\mathcal{O}_{g}\left(\frac{\al}{(4\la)^{1/3}}\right)$, and thus 
$$\la\in T(\al)\Longleftrightarrow \frac{\al}{(4\la)^{1/3}}\in \mathrm{PrePer}(g).$$
Hence, to find the elements of $T(\al)$ it suffices to find the preperiodic points of $g$, as in Proposition \ref{trichotomyg}.
\end{proof}

We now return to the proof of Proposition \ref{trichotomyg}. For this purpose, we will need the following lemmas, which exploit the fact that $0$ and $\infty$ are both $2-$adically attracting fixed points of the map $g$ with multipliers $-2$ and $4$ respectively, as we can see in the following remark.

\begin{remark}\label{Taylor}
\em{
For $z\in D(0,1)=\{z\in\bC_2~:~|z|<1\}$, we may write
\begin{align*}
g(z)=-2z-\frac{9z}{4}\displaystyle\sum_{\substack{n\ge 1}}(-4z^3)^n.
\end{align*}
Moreover, for $\phi(z)=\frac{1}{g(1/z)}=\frac{z^4+4z}{1-2z^3}\in\bC_2(z)$ and $z\in D(0,1)$ we have
\begin{align*}
\phi(z)=4z+\frac{9z}{2}\displaystyle\sum_{\substack{n\ge 1}}(2z^3)^n.
\end{align*}}
\end{remark}

\begin{lemma}\label{attracting}
If $w\in D(0,1)$, then as $n \to \infty$ both $g^n(w)\to 0$ and $\phi^n(w)\to 0$. In particular,
\begin{itemize}
\item
if $w\in\mathrm{PrePer}(g)$, then $g^m(w)=0$ for some $m\in\bN$.
\item
if $w\in\mathrm{PrePer}(\phi)$, then $\phi^k(w)=0$ for some $k\in\bN$.
\end{itemize}
\end{lemma}

\begin{proof}
In view of  Remark \ref{Taylor}, we have that if $w\in D(0,1)$, then $|g(w)|\le\max\left\{\frac{|w|}{2},|w|^4\right\}$ and $|\phi(w)|\le\max\left\{\frac{|w|}{4},|w|^4\right\}$. Thus, we infer that as $n\to \infty$ both $g^n(w)\to 0$ and $\phi^n(w)\to 0$. The rest of the statement now follows.
\end{proof}

\begin{remark}
\em{The above lemma follows from a more general fact about maps $f\in K(z)$ with good reduction and having an attracting fixed point, where $K$ is a local field, see \cite[Lemma 2.3]{Benedetto}. It implies that if $a\in\overline{K}$ is an attracting fixed point of $f$, then all the preperiodic points of $f$ that lie in the residue class of $a$ must map to $a$. }
\end{remark}

\begin{lemma}\label{fourthpower}
Let $n\in\bN$ and $w\in D(0,1)$. Then the following hold.
\begin{itemize}
\item
If $|g(w)|=|2^{1/3}|^{\frac{1}{4^n}}$, then $|w|^4=|2^{1/3}|^{\frac{1}{4^n}}$.
\item
If $|\phi(w)|=|4^{1/3}|^{\frac{1}{4^n}}$, then $|w|^4=|4^{1/3}|^{\frac{1}{4^n}}$.
\end{itemize}
Furthermore, if $n\in\bN$ is the smallest integer such that $g^{n+1}(w)=0$, then $|w|^{4^n}=|2^{1/3}|$. Similarly, if $n\in\bN$ is the smallest integer such that $\phi^{n+1}(w)=0$, then $|w|^{4^n}=|4^{1/3}|$.
\end{lemma}

\begin{proof}
Let $n\in\bN$ and $w\in D(0,1)$. Using our hypothesis and the Taylor expansion in Remark \ref{Taylor}, we have 
\begin{align*}
|g(w)|=\left|-2w-\frac{9w}{4}\displaystyle\sum_{\substack{n\ge 1}}(-4w^3)^n\right|=|2^{1/3}|^{\frac{1}{4^n}}.
\end{align*}
If $|w^3|\le |2|$, using the ultrametic inequality, we infer that 
\begin{align*}
|2^{1/3}|^{\frac{1}{4^n}}\le\max \{|2w|, |9w^4|\}\le |2|^{4/3},
\end{align*}
which contradicts the fact that $|2|<1$. Therefore, we must have $|w^3|>|2|$. Another application of the ultrametric inequality now yields that
\begin{align*}
|g(w)|=|w|^4=|2^{1/3}|^{\frac{1}{4^n}},
\end{align*}
as claimed in the statement of the lemma. Now, if $n\in\bN$ is the smallest integer such that $g^{n+1}(w)=0$, then $(g^n(w))^3=2$,  and hence $|g^n(w)|=|2^{1/3}|$. Inductively, we get $|w|^{4^n}=|2^{1/3}|$.

The case of $\phi$ is similar. In view of Remark \ref{Taylor} and our hypothesis we have 
\begin{align*}
|\phi(w)|=\left|4w+\frac{9w}{2}\displaystyle\sum_{\substack{n\ge 1}}(2w^3)^n\right|=|4^{1/3}|^{\frac{1}{4^n}}.
\end{align*}
If $|w^3|\le |4|$, the ultrametric inequality yields $|4^{1/3}|^{\frac{1}{4^n}}\le\max\{|4w|, |9w^4|\}\le |4|^{4/3}$, contradicting the fact that $|4|<1$. Thus, $|w^3|>|4|$ and $|\phi(w)|=|w|^4=|4^{1/3}|^{\frac{1}{4^n}}$. Now, if $n\in\bN$ is the smallest integer such that $\phi^{n+1}(w)=0$, then $(\phi^n(w))^3=-4$, and hence $|\phi^n(w)|=|4^{1/3}|$. Inductively, we get $|w|^{4^n}=|4^{1/3}|$. This finishes the proof of the lemma.
\end{proof}

We can now piece together the previous lemmas to prove Proposition 2.7.

\begin{proof}[Proof of Proposition \ref{trichotomyg}]
Let $w\in\mathrm{PrePer}(g)$. We consider the cases $|w|=1$, $|w|<1$ and $|w|>1$ separately. 

If $|w|=1$, then the ultrametric inequality yields that $|g^n(w)|=1$ for all $n\in\bN$ and in particular $0,\infty\notin\mathcal{O}_{g}(w)$. 

If $|w|>1$, then $z=\frac{1}{w}\in\mathrm{PrePer}(\phi)\cap D(0,1)$ and hence Lemma \ref{attracting} yields $\phi^{m+1}(z)=0$ for some $m\in\bN$. This, by using Lemma \ref{fourthpower}, implies $|z|^{4^m}=|4^{1/3}|$ and hence $|w|=|4^{-1/3}|^{\frac{1}{4^m}}$. If, in particular, we have $m=0$, then we immediately get $w^3=-\frac{1}{4}$.

Finally, assume $|w|<1$. By Lemma \ref{attracting} we have that $g^{m+1}(w)=0$ for some $m\in\bN$. Lemma \ref{fourthpower} now yields $|w|=|2^{1/3}|^{\frac{1}{4^m}}$. In the case $m=0$ we get $w^3=2$.
The proposition is now established.
\end{proof}

We conclude this section with some related remarks.
\begin{remark}\label{moreattractingpoints}
\em{We find it worthwhile to mention that $g$ has three other $2-$adically attracting fixed points, namely $-1$, $-\xi$, and $-\xi^2$, each with multiplier $-2$, where $\xi$ is a cube root of unity. These points give information about the preperiodic points of $g$ of flavor similar to the cases of $0$ and $\infty$. More specifically, for $w\in\mathrm{PrePer}(g)$, the following hold.
\begin{itemize}
\item
If $w\in D(-1,1)$, then there exists an $m\in\bN$ such that $g^{m+1}(w)=-1$. For the smallest such $m$, we have the equality $|w+1|=|2^{1/3}|^{\frac{1}{4^m}}$.
\item
If $w\in D(-\xi,1)$, then there exists an $m\in\bN$ such that $g^{m+1}(w)=-\xi$. For the smallest such $m$, we have the equality $|w+\xi|=|2^{1/3}|^{\frac{1}{4^m}}$.
\item
If $w\in D(-\xi^2,1)$, then there exists an $m\in\bN$ such that $g^{m+1}(w)=-\xi^2$. For the smallest such $m$, we have the equality $|w+\xi^2|=|2^{1/3}|^{\frac{1}{4^m}}$.
\end{itemize}
The proof follows along the same lines as Proposition \ref{trichotomyg}. We will briefly sketch the case of $-1$. For $z\in D(-1,1)$, we have
$$|g(z)+1|=|-2(z+1)-6(z+1)^2-16(z+1)^3-43(z+1)^4+R(z+1)|,$$
where $|R(z+1)|\le |z+1|^5$.
This implies that if $w\in\mathrm{PrePer}(g)\cap D(-1,1)$, then there exists a smallest $m\in\bN$ such that $g^{m+1}(w)=-1$, which in turn yields $|w+1|^4=|2^{1/3}|^{\frac{1}{4^n}}$. For the later, notice that if $|g(w)+1|\ge|2^{1/3}|$ then the ultrametric inequality yields $|g(w)+1|=|w+1|^4$. 
}
\end{remark}

\begin{remark}\label{periodicpoints}
\em{ Notice that $g$ has infinitely many periodic points. One way to see this is by recalling that $g$ is the Latt\`es map corresponding to the duplication map on $E:y^2=x^3+\frac{1}{4}$ and hence its periodic points are the $x-$coordinates of the points in $\displaystyle\cup_{n\in\bN_{\ge 1}}E[2^n-1]$.

Moreover, if $a\in\bC_2\setminus\{0,-1,-\xi,-\xi^2,\infty\}$ is a periodic point of $g\in\bC_2(z)$, then $|a|=|a+1|=|a+\xi|=|a+\xi^2|=1$. 
To see this, note that otherwise by Lemma \ref{trichotomy} and Remark  \ref{moreattractingpoints}, the orbit of $a$ under the action of $g$ meets a fixed point of $g$, contradicting the periodicity of $a$.
}
\end{remark}

\begin{remark}
\em{ The only $\bQ_2-$preperiodic points of $g\in\bC_2(z)$ are the $\bQ_2-$fixed points of $g$, that is $0$, $\infty$, and $-1$. To see this note that if $z\in\bQ_2\cap\mathrm{PrePer}(g)\setminus\{0,\infty,-1\}$ then either $|z|<1$ or $|z|>1$ or $|z+1|<1$, in which case Lemma \ref{fourthpower} and Remark \ref{moreattractingpoints} yield that there exists an $n\in\bN$ such that $|z|=|2^{1/3}|^{\frac{1}{4^n}}$ or $|z|=|4^{-1/3}|^{\frac{1}{4^n}}$ or $|z+1|=|2^{1/3}|^{\frac{1}{4^n}}$ respectively, contradicting the fact that $z\in\bQ_2$.}
\end{remark}

\begin{remark}\label{allareassumed}
\em{Observe that all the absolute values for $\la\in T(\al)$ that appear in Theorem \ref{trichotomy} do indeed occur, from which it immediately follows that $T(\al)$ is an infinite set. To see this, it suffices to prove that all absolute values that appear in Proposition \ref{trichotomyg} for preperiodic points of $g$ do indeed occur. As we have seen, $-1$ is a fixed point of $g$ of absolute value $1$. Let $n\in\bN$. To find $w\in\mathrm{PrePer(g)}$ such that $|w|=|2^{1/3}|^{\frac{1}{4^n}}$, in view of Lemma \ref{fourthpower}, it suffices to find $w\in\bC_2$ such that $g^{m+1}(w)=0$ and $g^m(w)\neq 0$. This can be achieved for $w\in\bC_2$ satisfying $(g^m(w))^3=2$. Analogously, to find $w\in\mathrm{PrePer(g)}$ such that $|w|=|4^{-1/3}|^{\frac{1}{4^n}}$, by Lemma \ref{fourthpower}, it suffices to find $z\in\bC_2$ such that $\phi^{m+1}(z)=0$ and $\phi^m(z)\neq 0$. This can be achieved for $z\in\bC_2$ satisfying $(\phi^m(z))^3=-4$. }
\end{remark}

\subsection{Some applications: impossible intersections}\label{2}

Let $\al,\beta\in\bC_2$. Assuming the existence of $\la\in T(\al)\cap T(\beta)$, Theorem \ref{trichotomy} allows us to compute an explicit list for the possible values of $\frac{|\al|}{|\beta|}$. 

\begin{corollary}\label{diftwoadic}
Assume that $T(\alpha)\cap T(\beta)\neq\emptyset$ and let 
\begin{align*}
X=\left\{1,2^{\frac{1}{4^r}},2^{\frac{1}{3\cdot4^r}},2^{\frac{1}{3}(\frac{1}{4^r}-\frac{1}{4^s})},2^{\frac{2}{3\cdot4^r}},2^{\frac{2}{3}(\frac{1}{4^r}-\frac{1}{4^s})},4^{\frac{1}{3\cdot4^r}}2^{\frac{1}{3\cdot4^s}}~:~r,s\in\bN,r\neq s\right\}.
\end{align*}
Then we have that either $\frac{|\al|}{|\beta|}\in X$ or $\frac{|\beta|}{|\al|}\in X$.
Moreover, $\frac{|\al|}{|\beta|}=\frac{1}{2}$ or $\frac{|\al|}{|\beta|}=2$ if and only if $\frac{\al^3}{\beta^3}=-8$ or $\frac{\al^3}{\beta^3}=-\frac{1}{8}$ respectively.
\end{corollary}

\begin{proof}
The proof follows immediately from Theorem \ref{trichotomy}.
\end{proof}

A consequence now is that $T(\al)\cap T(\beta)=\emptyset$, for all $\al,\beta$ that `disagree' with our list in Corollary \ref{diftwoadic}. More specifically, we get the following theorem.

\begin{theorem}\label{difempty}
If $\al,\beta\in\overline{\bQ_2}$ satisfy $\gcd\left(6,e(\bQ_2(\frac{\al}{\beta})|\bQ_2)\right)=1$, $\left|\frac{\al}{\beta}\right|\neq 1$ and $\frac{\al^3}{\beta^3}\notin\{-8,-\frac{1}{8}\}$, then we have that $T(\al) \cap T(\beta)=\emptyset$. Moreover, $T(a) \cap T(-2a)=\{-a^3\}$ for all $a\in\bC_2\setminus\{0\}$.
\end{theorem}

\begin{proof}
The proof follows combining the fact that for any $c\in\overline{\bQ_2}$ if $e:=e(\bQ_2(c)|\bQ_2)$, then $|c|\in 2^{\frac{\bZ}{e}}$ with Corollary \ref{diftwoadic}. To see that $T(a)\cap T(-2a)=\{-a^3\}$, note that by Theorem \ref{trichotomy} we get that if $\la\in T(a)\cap T(-2a)$ then $\la=-a^3$, in which case we have $f_{-a^3}(a)=\infty$ and $f_{-a^3}(-2a)=0$.
\end{proof}

\section{A Weierstrass family: more impossible intersections}\label{2}

We use our notation as in Section \ref{1}. Recall that by $|\cdot|$, we mean the $2-$adic absolute defined on $\bC_2$.
Theorem \ref{difempty} raises the question whether we could describe $T(\al)\cap T(\beta)$ for $\al,\beta\in\mathbb{C}_2$ with equal $2-$adic absolute values. In this section, towards partially answering this question, we aim to prove Theorem \ref{mainthm}, which asserts that if we restrict our attention to $\al,\beta\in\overline{\bQ_2}$ satisfying $\frac{\al}{\beta}\in\bQ$, then there are no parameters $\la$ such that both $\al,\beta$ are preperiodic for $f_{\la}$, unless $\frac{\al}{\beta}\in\{-2,-\frac{1}{2}\}$.
 
Before we state the main theorem of this section, a couple of remarks are in order.

\begin{remark}\label{automorphisms}
\em{The isotriviality of $f(z)\in\bC_2(t)(z)$ implies that for all $\al,\la,z\in\bC_2$, we have $f_{\la\al^3}(\al z)=\al f_{\la}(z)$. It easily follows that $\mathcal{O}_{f_{\la\al^3}}(\al)=\al\mathcal{O}_{f_{\la}}(1)$ and $T(\al)=\al^3 T(1)$.}
\end{remark}

\begin{remark}\label{1,c}
\em{From Remark \ref{automorphisms}, we get  
\begin{align*}
\la\in T(\al)\cap T(\beta)\Leftrightarrow\frac{\la}{\beta^3}\in T(1)\cap T\left(\frac{\al}{\beta}\right).
\end{align*}
 In particular, $\#(T(\al)\cap T(\beta))=\#(T(1)\cap T\left(\frac{\al}{\beta}\right))$.}
\end{remark}

For the following we fix an embedding $\overline{\bQ}\hookrightarrow\overline{\bQ}_2$.
For the reader's convenience, we now restate Theorem \ref{mainthm}.

\begin{theorem}\label{mainthmr}
If $\al,\beta\in\overline{\bQ_2}\setminus\{0\}$ are such that $\frac{\al}{\beta}\in\bQ\setminus\{-2,-\frac{1}{2}\}$, then $T(\al)\cap T(\beta)=\emptyset$. Moreover, for all $a\in\overline{\bQ_2}\setminus\{0\}$ we have $T(a) \cap T(-2a)=\{-a^3\}$.
\end{theorem}

As we have already seen in Theorem \ref{difempty}, when $\al,\beta\in\overline{\bQ_2}$ with $\frac{\al}{\beta}\in\bQ\setminus\{-2,-\frac{1}{2}\}$ and $|\al|\neq |\beta|$, we have $T(\al)\cap T(\beta)=\emptyset$.  Moreover when $a\in\overline{\bQ_2}\setminus\{0\}$, we have $T(a) \cap T(-2a)=\{-a^3\}$. Therefore, to prove Theorem \ref{mainthmr}, it suffices to show that $T(\al)\cap T(\beta)=\emptyset$ when $|\al|=|\beta|$. Our strategy will be to first show in Lemma \ref{toempty} that if $\la\in T(\al)\cap T(\beta)$, then either $0\in\mathcal{O}_{f_{\la}}(\al)\cap\mathcal{O}_{f_{\la}}(\beta)$ or $\infty\in\mathcal{O}_{f_{\la}}(\al)\cap\mathcal{O}_{f_{\la}}(\beta)$. 
Then, after proving the coprimality of certain polynomials in Lemmas \ref{primeexists} and \ref{1,-1}, we will rule out these two cases as well.

\begin{lemma} \label{toempty}
Let $\al,\beta\in\mathbb{C}_2$ with $|\al|=|\beta|$ and $|\al-\beta|\le\frac{|\al|}{2}$. Consider $\la\in T(\al)\cap T(\beta)$. Then either $0\in\mathcal{O}_{f_{\la}}(\al)\cap\mathcal{O}_{f_{\la}}(\beta)$ or $\infty\in\mathcal{O}_{f_{\la}}(\al)\cap\mathcal{O}_{f_{\la}}(\beta)$.
\end{lemma}

\begin{proof}
Let $\la\in T(\al)\cap T(\beta)$ and assume that $0,\infty\notin\mathcal{O}_{f_{\la}}(\al)\cap\mathcal{O}_{f_{\la}}(\beta)$. We want to show that this leads to a contradiction. In light of Theorem \ref{trichotomy} and the fact that $|\al|=|\beta|$ we get that $0\in\mathcal{O}_{f_{\la}}(\al)$ (respectively $\infty\in\mathcal{O}_{f_{\la}}(\al)$)  if and only if $0\in\mathcal{O}_{f_{\la}}(\beta)$ (respectively $\infty\in\mathcal{O}_{f_{\la}}(\beta)$). Our assumption thus implies that $0,\infty\notin\mathcal{O}_{f_{\la}}(\al)\cup\mathcal{O}_{f_{\la}}(\beta)$. For the rest of this proof we will denote $f^{n}_{\la}(\al)$ and $f^n_{\la}(\beta)$ by $t_n$ and $u_n$ respectively.

Since $\la \in T(\al)\cap T(\beta)$, we know that the sets
\begin{align*}
S=\{t_n~:~n\in\bN\}, \text{ and } T=\{u_n~:~n\in\bN\}
\end{align*}
are finite. Therefore, the set $M=\{|t_n-u_n|~:~n\in\bN\}$ is also finite.

We claim that $|t_n-u_n|=\frac{|\al-\beta|}{2^{n}}$ for all $n\in\bN$. This will contradict the fact that $M$ is finite, thus finishing our proof.
We will now prove the claim using induction.
For $n=0$, we have $|t_0-u_0|=|\al-\beta|$. For the inductive step, assume that the statement holds for some $n\ge 0$. Then 
\begin{align*}
|t_{n+1}-u_{n+1}|&=4\left|\frac{t_n^4-8\la t_n}{t_n^3+\la}-\frac{u_n^4-8\la u_n}{u_n^3+\la}\right|\\
&=4\left|\frac{(t_n^4-8\la t_n)(u_n^3+\la)-(u_n^4-8\la u_n)(t_n^3+\la)}{(u_n^3+\la)(t_n^3+\la)}\right|\\
&=\frac{4}{|u_n^3+\la||t_n^3+\la|}\left|t_n^3u_n^3(t_n-u_n)-\la(u_n^4-t_n^4)+8\la t_nu_n(t_n^2-u_n^2)+8\la^2(u_n-t_n)\right|.
\end{align*}
By the induction hypothesis, we know that $|t_n-u_n|=\frac{|\al-\beta|}{2^{n}}\le\frac{|\al|}{2^{n+1}}$. Moreover, since $\la\in T(\al)\cap T(\beta)$ and $0,\infty\notin\mathcal{O}_{f_{\la}}(\al)\cup\mathcal{O}_{f_{\la}}(\beta)$, we also have $\lambda\in T(t_n)\cap T(u_n)$ and $0,\infty\notin\mathcal{O}_{f_{\la}}(t_n)\cup\mathcal{O}_{f_{\la}}(u_n)$ for all $n\in\bN$. Theorem \ref{trichotomy} now yields that for all $n\in\bN$, $|t_n|=|u_n|=\sqrt[3]{\frac{|\la|}{4}}=|\al|$. Therefore, we get that
\begin{align*}
|u_n^3+\la|&=|t_n^3+\la|=4|\al|^3,\\
|t_n+u_n|&=|t_n-u_n+2u_n|\le\frac{|\al|}{2},\text{ and }\\
|t_n^2+u_n^2|&=|(t_n+u_n)^2-2t_nu_n|=\frac{|\al|^2}{2}.
\end{align*}
Hence, 
\begin{align*}
|t_n^3u_n^3(t_n-u_n)|&=|\al|^6|t_n-u_n|,\\
|\la(u_n^4-t_n^4)|&\le|\al|^6|t_n-u_n|,\\
|8\la t_nu_n(t_n^2-u_n^2)|&\le\frac{|\al|^6}{4}|t_n-u_n|,\\
|8\la^2(t_n-u_n)|&=2|\al|^6|t_n-u_n|.
\end{align*}
Thus, using the induction hypothesis, we obtain
\begin{align*}
|t_{n+1}-u_{n+1}|=\frac{1}{2}|t_n-u_n|=\frac{|\al-\beta|}{2^{n+1}}.
\end{align*}
This establishes our claim and concludes the proof.
\end{proof}

\begin{remark}\label{rationalassumption}
\em{We point out here that the condition $\frac{\al}{\beta}\in\bQ$ in Theorem \ref{mainthmr} has been made to ensure that when  $\al,\beta$ have equal $2-$adic absolute values, then $|\al-\beta|\le\frac{|\al|}{2}$ holds. In this case we can apply Lemma \ref{toempty}.}
\end{remark}

To proceed with our proof, we need the following definition.

\begin{definition}\label{def:recursion}
\em{Given $a\in\mathbb{C}\setminus\{0\}$, we write $f_t^n(a)=\frac{A_n(a,t)}{B_n(a,t)}$, where $A_n(a,t),B_n(a,t)\in\mathbb{C}[t]$ are polynomials given recursively as 
\begin{align*}
&A_0(a,t)=a, B_0(a,t)=1.\\
&A_{n+1}(a,t)=A_{n}(a,t)^4-8tA_n(a,t)B_n(a,t)^3.\\
&B_{n+1}(a,t)=4B_n(a,t)A_n(a,t)^3+4tB_n(a,t)^4.
\end{align*}}
\end{definition}

\begin{lemma}\label{coprime}
Let $a\in\mathbb{C}\setminus\{0\}$. We have $\gcdd(A_n(a,t),B_n(a,t))=1$ for all $n\in\bN$. Moreover, $\deg_{t}(A_n(a,t))=\deg_{t}(B_n(a,t))=\frac{4^n-1}{3}$ for all $n\in\bN$.
\end{lemma}

\begin{proof}
The proof follows from an easy induction. Note that our hypothesis that $a\neq 0$ implies that $A_n(a,0),B_n(a,0)\neq 0$ and thus $t\nmid A_n(a,t),B_n(a,t)$  for all $n\in\bN$.
\end{proof}

Our aim in the sequel will be to prove that when $a,b\in\bZ$, the polynomials $A_n(a,t)$ and $A_n(b,t)$ (respectively $B_n(a,t)$ and $B_n(b,t)$) are coprime. This will in turn imply that $0,\infty\notin\mathcal{O}_{f_{\la}}(a)\cap\mathcal{O}_{f_{\la}}(b)$, which combined with view of Remark \ref{1,c} is what we need in order to conclude the proof of Theorem \ref{mainthmr}, when $\frac{\al}{\beta}=\frac{a}{b}$. To achieve this we will first establish a few key lemmas.

\begin{lemma}\label{primeexists}
For all $a,b\in\bZ$ for which there exists a prime $p\neq 2$ such that $p|a$ and $p\nmid b$, we have 
\begin{align*}
\gcd(A_n(a,t),A_n(b,t))=\gcd(B_n(a,t),B_n(b,t))=1, \text{ for all }n\in\bN.
\end{align*}
\end{lemma}

\begin{proof}
We start by noting that given $g(t)\in\bZ[t]$, we denote its reduction modulo $p$ by $\overline{g}(t)\in\mathbb{F}_p[t]$, where $p$ is the same prime as in the statement of the lemma. Moreover, we denote $a_n(a,t)=\frac{A_n(a,t)}{a}$. From the recursion in Definition \ref{def:recursion} we see that $a_n(a,t)\in\bZ[t]$. Our strategy is to first establish that 
\begin{align}\label{eqn:gcd a_n at 2 different values}
\gcd\left(\overline{a_n}(a,t),\overline{A_n}(b,t)\right)=\gcd(\overline{B_n}(a,t),\overline{B_n}(b,t))=1. 
\end{align}
Since $a_n(a,t)\in \bZ[t]$ and $p|a$, we get that $\overline{A_n}(a,t)=0$ for all $n\in\bN$. Therefore, the recursive relations in Definition \ref{def:recursion} yield that
\begin{align*}
\overline{B_{n+1}}(a,t)=4t\overline{B_n}(a,t)^4\text{ and } \overline{B_{1}}(a,t)=4t.
\end{align*}
Thus, we obtain that $\overline{B_{n}}(a,t)=(4t)^{\frac{4^{n}-1}{3}}$.
Additionally, we have the following equalities.
\begin{align}\label{polynomialredrec}
\overline{a_{n+1}}(a,t)=-2(4t)^{4^n}\overline{a_n}(a,t)\text{ and }\overline{a_{1}}(a,t)&=-8t.
\end{align}
From \eqref{polynomialredrec}, it follows that $\overline{a_{n}}(a,t)=(-2)^{n}(4t)^{\frac{4^{n}-1}{3}}$.

Observe now that, to derive \eqref{eqn:gcd a_n at 2 different values}, it suffices to prove that $\overline{A_n}(b,0),\overline{B_n}(b,0)\neq 0$. Notice that $A_n(b,0)=b^{4^n}$ and $B_n(b,0)=4^nb^{4^n-1}$. Since $p\neq 2$ and $p\nmid b$, we have established the equality in \eqref{eqn:gcd a_n at 2 different values}.
Now note that the degrees of the polynomials $a_n(a,t),A_n(b,t),B_n(a,t)$ and $B_n(b,t)$, as computed in Lemma \ref{coprime}, are equal to the degrees of their reductions modulo $p$. This, combined with \eqref{eqn:gcd a_n at 2 different values} finishes the proof of the lemma.
\end{proof}

We will prove another useful lemma of a similar flavor.

\begin{lemma}\label{1,-1}
For all $n\in\bN$, $\gcd(A_n(1,t),A_n(-1,t))=\gcd(B_n(1,t),B_n(-1,t))=1$.
\end{lemma}
\begin{proof}
We proceed in a similar fashion to the proof of the previous lemma, just that this time we reduce polynomials modulo 3. Throughout this proof we write $\overline{g}(t)\in\mathbb{F}_3[t]$ for the reduction modulo $3$ of $g(t)\in\mathbb{Z}[t]$. We will show that 
\begin{align}\label{coprimemod3}
\gcd\left(\overline{A_n}(1,t),\overline{A_n}(-1,t)\right)=\gcd\left(\overline{B_n}(1,t),\overline{B_n}(-1,t)\right)=1.
\end{align}
Inductively, we can prove that $\overline{A_{n}}(1,t)=\overline{B_{n}}(1,t)\text{ for all }n\in\bN$.
Therefore, using the recursion for $A_n(1,t)$, we get that $\overline{A_{n+1}}(1,t)=\overline{A_{n}}(1,t)^4(1+t)$.
Thus, 
\begin{align*}
\overline{A_{n}}(1,t)=\overline{B_{n}}(1,t)=(1+t)^{\frac{4^n-1}{3}}.
\end{align*}
Moreover, inductively we can see that
\begin{align*}
&\overline{A_{n}}(-1,t)=(1-t)^{\frac{4^n-1}{3}},\\
&\overline{B_{n}}(-1,t)=-(1-t)^{\frac{4^n-1}{3}}.
\end{align*}
This establishes \eqref{coprimemod3}. Since the degrees of the polynomials $A_n(1,t),A_n(-1,t),B_n(1,t)$ and $B_n(-1,t)$, computed in Lemma \ref{coprime}, are equal to the degrees of their reductions modulo $3$, \eqref{coprimemod3} yields the lemma.
\end{proof}

We are now ready to prove Theorem \ref{mainthmr}. In the following, for a non-zero integer $n$ and a prime $p$, we write $${\rm exp}_p(n)=\max\{e\in\bN~:~p^e\vert n\}.$$

\begin{proof}[Proof of Theorem \ref{mainthmr}.]

Let $\al,\beta\in\overline{\bQ_2}$ be such that $\frac{\al}{\beta}\in\bQ\setminus\{-2,-\frac{1}{2}\}$. We aim to prove that $T(\al)\cap T(\beta)=\emptyset$. In view of Theorem \ref{difempty}, we have that $T(\al)\cap T(\beta)=\emptyset$ if $|\al|\neq|\beta|$, unless $\frac{\al}{\beta}\notin\{-2,-\frac{1}{2}\}$. Assume now $|\al|=|\beta|$.  We may write $\frac{\al}{\beta}=\frac{a}{b}$, where $a,b\in\bZ$ are coprime and satisfy $|a|=|b|=1$. By Remark \ref{1,c} it suffices to prove that $T(a)\cap T(b)=\emptyset$. Assume to the contrary that $\la\in T(a)\cap T(b)$. Now Lemma \ref{toempty} and Remark \ref{rationalassumption} together imply that either $0\in\mathcal{O}_{f_{\la}}(a)\cap\mathcal{O}_{f_{\la}}(b)$ or $\infty\in\mathcal{O}_{f_{\la}}(a)\cap\mathcal{O}_{f_{\la}}(b)$. 

Observe that $0\in\mathcal{O}_{f_{\la}}(a)\cap \mathcal{O}_{f_{\la}}(b)$ if and only if $A_n(a,\la)=A_n(b,\la)=0$ for some $n\in\bN$, and $\infty\in\mathcal{O}_{f_{\la}}(a)\cap \mathcal{O}_{f_{\la}}(b)$ if and only if $B_n(a,\la)=B_n(b,\la)=0$ for some $n\in\bN$. Hence, to prove the theorem it suffices to show that
\begin{align}\label{gcd}
\gcd(A_n(a,t),A_n(b,t))=\gcd(B_n(a,t),B_n(b,t))=1 \text{ for all }n\in\bN.
\end{align}

To this end, we will consider two cases. If $a=-b=1$, on invoking Lemma \ref{1,-1} we see that \eqref{gcd} follows. If on the other hand there exists a prime $p$ such that ${\rm exp}_p(a)\neq{\rm exp}_p(b)$, since $|a|=|b|=1$, we see that $p\neq 2$ and \eqref{gcd} again holds true by Lemma \ref{primeexists}.
\end{proof}

An immediate corollary of Theorem \ref{mainthmr} is the following.

\begin{corollary}
If $\al\in\bQ$ and $\la\in T(\al)$, the orbit of $\al$ under the action of $f_{\la}$ does not contain any rational number other than $\al$ except possibly $0$ and $\infty$.
\end{corollary}
\begin{proof}
Let $\al\in\bQ$ and $\la\in T(\al)$. Assume that for some $n\in\bN$ we have $f^n_{\la}(\al)\in\bQ\setminus\{0,\infty,\al\}$ to derive a contradiction. Observe that $\la\in T(\al)\cap T(f^n_{\la}(\al))$. This observation combined with Theorem \ref{mainthmr} now yields that one of the following is true.

$\bullet$ $f^n_{\la}(\al)=-2\al$, in which case $\la=-\al^3$ and $f_{\la}(\al)=\infty$ contradicting the fact that $f^n_{\la}(\al)\neq\infty$.

$\bullet$ $\al=-2f^n_{\la}(\al)$, in which case $\la=\frac{\al^3}{8}$ and $f_{\la}(\al)=0$ contradicting the fact that $f^n_{\la}(\al)\neq 0$.

In each case we derived a contradiction, yielding our claim.
\end{proof}

We conclude this section with some observations and further questions.

\begin{remark}
\em{Let $h:\overline{\bQ}\to \bR_{\ge 0}$ denote the absolute logarithmic Weil height, as in \cite[Section 3.1]{Silverman07}. As seen from \eqref{goodreduction} we have $\la\in T(1)\Leftrightarrow \frac{1}{(4\la)^{1/3}}\in\mathrm{PrePer}(g)$. Combining this with \cite[Theorem 3.12]{Silverman07} we see that there is a constant $M>0$ such that $h(\la)<M$ for all $\la\in T(1)$. Moreover, by Remark \ref{automorphisms} we have $T(\al)\cap T(\beta)\neq\emptyset$ if and only if there exist $\la,\mu\in T(1)$ such that $\al^3\la=\beta^3\mu$. Therefore, $3h(\frac{\al}{\beta})=h(\frac{\mu}{\la})\le 2M$. Hence, there is an absolute constant $C>0$ such that $h\left(\frac{\al}{\beta}\right)<C$ for all $\al,\beta\in\overline{\bQ}$  that satisfy $T(\al)\cap T(\beta)\neq\emptyset$. In particular there exist only finitely many $\frac{\al}{\beta}$ of bounded degree such that $T(\al)\cap T(\beta)\neq\emptyset$. In Theorem \ref{mainthmr}, we show that if $\frac{\al}{\beta}\in\bQ$, then $T(\al)\cap T(\beta)\neq\emptyset$ if and only if $\frac{\al}{\beta}\in\{-2,-\frac{1}{2}\}$. This raises the following natural question: Fix $d\in\bZ$, what is the best upper bound (depending on $d$) on the number of $\frac{\al}{\beta}$ of degree at most $d$ such that $T(\al)\cap T(\beta)\neq\emptyset$?}
\end{remark}

Finally, it would be interesting to know whether there exist $\al,\beta\in\overline{\bQ}$ such that $2\le\#(T(\al)\cap T(\beta))<+\infty$. We note here that by Remark \ref{automorphisms} if $\al^3=\beta^3$, we have $T(\al)\cap T(\beta)=T(\al)=T(\beta)$, which by Remark \ref{allareassumed} is an infinite set. 

\section{The Legendre family of elliptic curves}\label{3}
For this section, let $E_{\la}:y^2=x(x-1)(x-\la)$ be the Legendre family of elliptic curves parametrized by $\la\in\mathbb{C}_2\setminus\{0,1\}$. 
The Latt\`{e}s map induced by the multiplication by $2$ map on $E_{\la}$, is given as 
\begin{align*}
f_{\la}(z)=\frac{(z^2-\la)^2}{4z(z-1)(z-\la)}.
\end{align*}
Let $\al\in\bC_2$. We define $T(\al)$ as follows.
\begin{align*}
T(\al)&=\{\la\in\mathbb{C}_2\setminus\{0,1\}~:~(\al,\sqrt{\al(\al-1)(\al-\la)})\in(E_{\la})_{tors}\}\\
&=\{\la\in\mathbb{C}_2\setminus\{0,1\}~:~\al\text{ is preperiodic for }f_{\la}\}.
\end{align*}

First, we prove the following easy proposition.

\begin{proposition}\label{reciprocal}
For all $n\in\bN$ and $\la\in\mathbb{C}$, we have $f^n_{\frac{1}{\la}}(\frac{1}{x})=\frac{1}{\la}f^n_{\la}(x)$. In particular, $\la\in T(x)$ if and only if $\frac{1}{\la}\in T(\frac{1}{x})$.
\end{proposition}

\begin{proof}
For $n=1$, we have $f_{\frac{1}{\la}}(\frac{1}{x})=\frac{1}{\la}f_{\la}(x)$.
Using the easily verifiable fact $f_{\frac{1}{\la}}(\frac{z}{\la})=\frac{1}{\la}f_{\la}(z)$, we get the result inductively. Alternatively, one can see that $\la\in T(x)$ if and only if $\frac{1}{\la}\in T(\frac{1}{x})$ by noting that the affine map $(x,y)\to (\frac{1}{x},\frac{y}{x^2\sqrt{\la}})$ extends to an isomorphism between the elliptic curves $E_{\la}$ and $E_{\frac{1}{\la}}$.
\end{proof}

The following theorem is a restatement of \cite[Theorem 3]{Stoll14}. Here we provide a different, shorter proof.

\begin{theorem}\label{stoll}
Let $\al\in\mathbb{C}_2\setminus\{0,1\}$ and $\la\in T(\al)\setminus\{\al\}$. If $|\al|\le 1$, then $|\al^2-\la|<1$. If $|\al|>1$, then $|\la|>1$.
\end{theorem}

\begin{proof}
Consider $\al\in\mathbb{C}_2$ with $|\al|\le 1$ and let $\la\in T(\al)\setminus\{\al\}$. First we will show that $|\la|\le 1$. Assume, so as to derive a contradiction, that $|\la|>1$. Then, we have
\begin{align*}
|f_{\la}(\al)|=\frac{4|\la|^2}{|\al||\al-1||\la|}\ge 4|\la|.
\end{align*}
Inductively, we get $|f^{n+1}_{\la}(\al)|=4|f^n_{\la}(\al)|$ for all $n\ge 1$. Since $\la\in T(\al)$, this implies $f_{\la}(\al)=\infty$ which in turn contradicts the fact that $\la\neq\al$.
Therefore, we must have $|\la|\le 1$. To prove $|\al^2-\la|<1$, let us assume the opposite and see what happens. Using the fact $|\la|\le 1$, we have 
\begin{align*}
|f_{\la}(\al)|=\frac{4|\al^2-\la|^2}{|\al||\al-1||\al-\la|}\ge 4.
\end{align*}
Again, inductively we get $|f^{n+1}_{\la}(\al)|=4|f^n_{\la}(\al)|$ for all $n\ge 1$. Since $\la\in T(\al)$, this yields $f_{\la}(\al)=\infty$, contradicting the fact that $\la\neq\al$. Hence $|\al^2-\la|<1$.

The second part of the statement now follows by the first part and Proposition \ref{reciprocal}.
\end{proof}

Therefore, exactly as in \cite[Corollary 4]{Stoll14}, denoting by $\rho$ the natural reduction map $\mathbb{P}^1(\bC_2)\to\mathbb{P}^1(\overline{\mathbb{F}_2})$, we get that

\begin{corollary}
If $\al,\beta\in\mathbb{C}_2\setminus\{0,1\}$ such that $\rho(\al)\neq\rho(\beta)$, then $T(\al)\cap T(\beta)\subset\{\al,\beta\}$.
\end{corollary}

For examples of $\al,\beta$ such that $T(\al)\cap T(\beta)$ is empty or has exactly one or two elements, we refer the reader to \cite[Example 5]{Stoll14}.

Now we aim to strengthen this result. In particular we provide an effective description of $T(\al)\cap T(\beta)$ even in some cases when $\rho(\al)=\rho(\beta)$. To this end, we have the following.

\begin{theorem}\label{2empty}
If $\al\in\mathbb{C}_2$ satisfies $|\al|\le \frac{1}{4}$, then $T(2)\cap T(\al)\subset\{\al\}$. Moreover, if $|\al|<\frac{1}{4}$, then $T(2)\cap T(\al)=\emptyset$.
\end{theorem}

\begin{proof}
Let $\al\in\bC_2$ be such that $|\al|\le\frac{1}{4}$ and $\la\in T(2)\cap T(\al)\setminus\{2,\al\}$. In view of Theorem \ref{stoll}, we know that $|\la|<1$. We claim, in fact, that the following is true.
\begin{claim}\label{lainT(2)}
If $\la\in T(2)\setminus\{2\}$, then $|\la|=\frac{1}{4}$.
\end{claim}
\begin{proof}
 Using an argument by contradiction, we will first show that $|\la|\ge\frac{1}{4}$. Assume that $|\la|< \frac{1}{4}$. Then $|f_{\la}(2)|=\frac{8|4-\la|^2}{|2-\la|}=1$ and therefore $|(f_{\la}(2))^2-\la|=1$. However, since $\la\in T(f_{\la}(2))\setminus\{f_{\la}(2)\}$, this contradicts Theorem \ref{stoll}.

 A similar argument also shows that $|\la|\le \frac{1}{2}$. Indeed, if $|\la|>\frac{1}{2}$, then we have $|f_{\la}(2)|=8|\la|>4$, which contradicts Theorem \ref{stoll} since $\la\in T(f_{\la}(2))\setminus\{f_{\la}(2)\}$ and $|\la|<1$.

Now we know that $\frac{1}{4}\leq |\la| \leq \frac{1}{2}$. Assume that $|\la|> \frac{1}{4}$, so as to derive a contradiction and establish the claim. Since $|\la|\le\frac{1}{2}$, we have $|f_{\la}(2)|=\frac{8|\la|^2}{|2-\la|}\ge 16|\la|^2>1$, which combined with the fact that $\la\in T(f_{\la}(2))\setminus\{f_{\la}(2)\}$ contradicts Theorem \ref{stoll}. This yields the claim.
\end{proof}
To finish the proof of the theorem, notice that Claim \ref{lainT(2)} yields that $|\la|=\frac{1}{4}$. This, combined with our assumption that $|\al|\le\frac{1}{4}$, implies $|f_{\la}(\al)|=\frac{4|\la|^2}{|\al||\al-\la|}\ge 4$, which by Theorem \ref{stoll} contradicts the fact that $\la\in T(f_{\la}(\al))\setminus\{f_{\la}(\al)\}$.
Therefore, we obtain that $T(2)\cap T(\al)\subset\{2,\al\}$. We can easily see that $2\notin T(\al)$, since $|\al|\le\frac{1}{4}$. Thus, in fact we get that $T(2)\cap T(\al)\subset\{\al\}$. If in particular $|\al|<\frac{1}{4}$, then by Claim \ref{lainT(2)} we have $T(2)\cap T(\al)=\emptyset$.
\end{proof}

Combining now Proposition \ref{reciprocal} and Theorem \ref{2empty}, we get the following.

\begin{corollary}\label{1/2empty}
If $\beta\in\mathbb{C}_2$ satisfies $|\beta|\ge4$, then $T(\frac{1}{2})\cap T(\beta)\subset\{\beta\}$. Moreover, if $|\beta|>4$, then $T(\frac{1}{2})\cap T(\beta)=\emptyset$.
\end{corollary}

\section{Other families}\label{4}
We will now consider
\begin{align*}
f_{\la}(z)=\frac{z^d+\la}{pz},
\end{align*}
where $d\ge 2$ and $p\in\bZ$ prime. 
Our method will give results of flavor similar to the results in Section \ref{3} for this family of rational maps. However, we find it worthwhile to mention that the above family is not a Latt\`es family. To see this note that for all $\la\in\bC$ the maps $f_{\la}$ have an attracting fixed point in the topology induced by the standard complex absolute value; when $d>2$ we have that $\infty$ is a fixed critical point and when $d=2$ the points $\pm\sqrt{\frac{\la}{p-1}}$ are fixed points with multiplier $\frac{2-p}{p}$. On the other hand, for Latt\`es maps all periodic points are repelling and dense in $\mathbb{P}^1(\bC)$, as illustrated by the fact that the Julia set of a Latt\`es map is the entire Riemann sphere \cite[Theorem 1.43]{Silverman07}. For a definition of the Julia set of a rational map we refer the reader to \cite{Silverman07}.

In this section $|\cdot|_p$ will denote the $p-$adic absolute value on $\bC_p$, with $|p|_p=\frac{1}{p}$. We write $T(\al)=\{\la\in\mathbb{C}_p~:~ \al\text{ is preperiodic for } f_{\la}\}$. We note that $0$ is a persistently preperiodic point for the family of rational maps $f_{\la}$ where $\la\in\bC_p$.
Therefore, in the following we consider $\al\neq 0$.

\begin{theorem}
Let $\al\in\mathbb{C}_p\setminus\{0\}$ with $|\al|_p\le1$ and let $\la\in T(\al)$. Then $|\al^d+\la|_p<1$. If on the other hand we have $|\al|_p>1$, then $|\la|_p>1$.
\end{theorem}

\begin{proof}
Consider $\al\in\mathbb{C}_p\setminus\{0\}$ with $|\al|_p\le 1$ and let $\la\in T(\al)$. First, we will show that $|\la|_p\le 1$.
Assume to the contrary that $|\la|_p>1$. Then, we have
\begin{align*}
|f_{\la}(\al)|_p=\frac{p|\la|_p}{|\al|_p}\ge p|\la|_p>p.
\end{align*}
Inductively, we get that $|f^{n+1}_{\la}(\al)|_p=p|f^n_{\la}(\al)|_p^{d-1}$ for all $n\ge 1$. Since $f_{\la}(\al)\neq\infty$ and $d\ge 2$, this contradicts the fact that $\la\in T(\al)$.
Therefore we must have $|\la|_p\le 1$. Next we prove that $|\al^d+\la|_p<1$.

Assume, for the sake of contradiction, that $\la\in T(\al)$ and $|\al^d+\la|_p\ge1$. Then we have $|f_{\la}(\al)|_p\ge p$. Inductively this implies $|f^{n+1}_{\la}(\al)|_p=p|f^n_{\la}(\al)|_p^{d-1}$ for all $n\ge 1$. Since $\la\in T(\al)$ and $d\ge 2$, this yields $f_{\la}(\al)=\infty$, contradicting the fact that $\al\neq0$. Hence $|\al^d+\la|_p<1$.

We will now prove the second part of the statement. Assume that $|\al|_p>1$ and that $\la\in T(\al)$. We will see that $|\la|_p>1$.
Assume, to the contrary, that $|\la|_p\le 1$. Then, $|f_{\la}(\al)|_p=p|\al|_p^{d-1}$ and inductively $|f^{n+1}_{\la}(\al)|_p=p|f^n_{\la}(\al)|_p^{d-1}$ for all $n\in\bN$. This, combined with $|\al|_p>1$, contradicts our assumption that $\la\in T(\al)$. Therefore, $|\la|_p>1$. The proof follows.
\end{proof}

Now, if we denote by $\rho$ the natural reduction map $\mathbb{P}^1(\bC_p)\to\mathbb{P}^1(\overline{\mathbb{F}_p})$, we get the following theorem.

\begin{theorem}
If $\al,\beta\in\mathbb{C}_p\setminus\{0\}$ are such that $\rho(\al^d)\neq\rho(\beta^d)$, then $T(\al)\cap T(\beta)=\emptyset$.
\end{theorem}

To highlight the difference between the results in this section and in Section \ref{1}, we find it worthwhile to point out the following. 

\begin{remark}
\em{Recall that in Section \ref{1}, the map $f(z)=\frac{z^4-8tz}{4(z^3+t)}\in\bC_2(t)(z)$ has the following properties: It is isotrivial, and it is conjugate to the map $g(z)=\frac{z^4-2z}{4z^3+1}$ with $2-$adic good reduction. For the map $f(z)=\frac{z^d+t}{pz}\in\bC_p(t)(z)$ in this section, these properties are true only when $(d,p)=(2,2)$, in which case for the choice $L(z)=2z-1$, we have
\begin{align*}
L^{-1}\circ g\circ L(z)=\frac{z^2}{2z-1}, \text{where }g(z)=\frac{z^2+1}{2z}
\end{align*}
and the map $\frac{z^2}{2z-1}$ has $2-$adic good reduction. If $d=2$ but $p\neq 2$, then the map $f(z)=\frac{z^2+t}{pz}\in\bC_p(t)(z)$ is still isotrivial. Indeed,  for  $M(z)=t^{1/2}z$, we have $M^{-1}\circ f\circ M(z)=\frac{z^2+1}{pz}\in\bC_p(z)$. However, $g(z)=\frac{z^2+1}{pz}\in\bC_p(z)$  in not $\mathrm{PGL}(2,\bC_p)-$conjugate to a map with good reduction, since it has two $p-$adically repelling fixed points $\pm\frac{1}{\sqrt{p-1}}$ with multiplier $\frac{2-p}{p}$, \cite[Theorem]{Benedetto potentially good}. Finally, if $d\ge 3$, then the map $f(z)=\frac{z^d+t}{pz}\in\bC_p(t)(z)$ is not isotrivial.}
\end{remark}

\section*{Acknowledgements}
I would like to thank Dragos Ghioca for suggesting the question and helping with exposition, and for insightful comments and enlightening conversations. I would also like to thank Rob Benedetto and Khoa Nguyen for insightful comments. Finally, I would like to thank the anonymous referee for various comments and suggestions that greatly improved this article.

\end{document}